\documentclass[12pt]{amsart}

\usepackage[utf8]{inputenc}
\usepackage[T1]{fontenc}
\usepackage{url,  mathrsfs}
\usepackage{amsmath}
\usepackage{amssymb}
\usepackage{amsthm}
\usepackage{graphicx, epstopdf}
\usepackage{subfig}
\usepackage{color}
\usepackage{bbm}
\usepackage{dsfont}
\usepackage{subfloat}
\usepackage[draft]{fixme}
\usepackage{bm}
\usepackage{bbm} 
\usepackage{ifpdf}
\usepackage{array}
\usepackage{mathtools}
\usepackage{multicol}
\usepackage{multirow}
\usepackage{graphicx}
\usepackage{tikz}
\usepackage{hyperref}
\usepackage[capitalize]{cleveref}

\newcommand{\A}{\mathcal{A}}
\newcommand{\D}{\mathbb{D}}
\newcommand{\Q}{\mathbb{Q}}
\newcommand{\N}{\mathbb{N}}
\newcommand{\Z}{\mathbb{Z}}
\newcommand{\R}{\mathbb{R}}

\newcommand{\C}{\mathbb{C}}

\newcommand{\z}{\mathbf{z}}

\newcommand{\lv}{\boldsymbol\ell}

\newcommand{\xv}{\mathbf{x}}

\newcommand{\btheta}{\boldsymbol\theta}
\newcommand{\bomega}{\boldsymbol\omega}
\newcommand{\balpha}{\boldsymbol\alpha}
\newcommand{\bbeta}{\boldsymbol\beta}

\DeclarePairedDelimiterX{\floor}[1]{\lfloor}{\rfloor}{#1}
\DeclarePairedDelimiterX{\ceil}[1]{\lceil}{\rceil}{#1}
\DeclarePairedDelimiterX{\card}[1]{\ellert}{\rvert}{#1}
\DeclarePairedDelimiterX{\abs}[1]{\ellert}{\rvert}{#1}
\DeclarePairedDelimiterX{\norm}[1]{\ellert}{\rVert}{#1}
\DeclarePairedDelimiterX{\tuple}[1]{\lparen}{\rparen}{#1}
\DeclarePairedDelimiterX{\parens}[1]{\lparen}{\rparen}{#1}
\DeclarePairedDelimiterX{\brackets}[1]{\lbrack}{\rbrack}{#1}
\DeclarePairedDelimiterX{\set}[1]\{\}{#1}
\let\Pr\relax
\DeclarePairedDelimiterXPP{\Pr}[1]{\mathbb{P}}[]{}{#1}
\DeclarePairedDelimiterXPP{\PrX}[2]{\mathbb{P}_{#1}}[]{}{#2}
\DeclarePairedDelimiterXPP{\Ex}[1]{\mathbb{E}}[]{}{#1}
\DeclarePairedDelimiterXPP{\ExX}[2]{\mathbb{E}_{#1}}[]{}{#2}

\DeclarePairedDelimiterX{\xt}[1]{\lbrack}{\rbrack}{#1}

\theoremstyle{plain}
\newtheorem{thm}{Theorem}[section]
\newtheorem{lem}[thm]{Lemma}
\newtheorem{prop}[thm]{Proposition}
\newtheorem{cor}[thm]{Corollary}

\newtheorem*{thm*}{Theorem} 

\theoremstyle{definition}
\newtheorem{Def}[thm]{Definition}
\newtheorem{exampleth}[thm]{Example}
\newenvironment{example}{\begin{exampleth}}{\hfill
    $\diamond$\\ \end{exampleth}}
\newtheorem*{exampleth*}{Example}
\newenvironment{example*}{\begin{exampleth*}}{\hfill
    $\diamond$ \end{exampleth*}}
\newtheorem{remark}[thm]{Remark}

\DeclareMathOperator{\im}{Im}

\DeclareMathOperator{\conv}{conv}

\usepackage[inner=1in,outer=1in,top=1in,bottom=1in]{geometry}

\makeatletter
\newcommand{\addresseshere}{%
	\enddoc@text\let\enddoc@text\relax
}
\makeatother

\begin{document}
\title[Lee-Yang polynomials from real-rooted exponential sums]{Every real-rooted exponential polynomial is the restriction of a Lee-Yang polynomial}

\author{Lior Alon}
\address{Massachusetts Institute of Technology, Cambridge, MA, U.S.A }
\email{lioralon@mit.edu}

\author{Alex Cohen}
\address{Massachusetts Institute of Technology, Cambridge, MA, U.S.A }
\email{alexcoh@mit.edu}

 \author{Cynthia Vinzant}
\address{University of Washington, Seattle, WA, U.S.A.}
\email{vinzant@uw.edu}

\maketitle

\begin{abstract}
A Lee-Yang polynomial $ p(z_{1},\ldots,z_{n}) $ is a polynomial that has no zeros in the polydisc $ \D^{n} $ and its inverse $ (\C\setminus\overline{\D})^{n} $. We show that any real-rooted  exponential polynomial of the form $f(x) = \sum_{j=0}^s c_j e^{\lambda_j x}$ can be written as the restriction of a Lee-Yang polynomial to a positive line in the torus. 
Together with previous work by Olevskii and Ulanovskii, this implies that the Kurasov-Sarnak construction of $ \N $-valued 
Fourier quasicrystals from stable polynomials comprises every possible $ \N $-valued Fourier quasicrystal. 
\end{abstract}

\section{Introduction}
	
A multivariate polynomial $p  \in \C[z_1, \hdots, z_n]$ is called \emph{Schur stable} if it has no zeros in the product $\D^n$ of the open unit disk, $\D = \{z\in \C : |z|<1\}$. 
Following \cite{Ruelle2010}, we call $p$ a \emph{Lee-Yang polynomial} if it has no zeros in $\Omega^n$ both for $\Omega = \D$ and $\Omega = \{z\in \C : |z|>1\}$. Equivalently,  $ p $ is a Lee-Yang polynomial if both $p$ and $p^{\dagger} = \prod_{j=1}^n z_j^{\deg_j(p)}p(1/z_1, \hdots, 1/z_n)$ are Schur stable, where for each $j$,  $ \deg_j(p) $ denotes the degree of $p$ in the variable $ z_{j} $.
	
For any Lee-Yang polynomial $p(z_1, \hdots, z_n)= \sum_{\balpha}c_{\balpha}z_1^{\alpha_1}\cdots z_n^{\alpha_n}$ and  vector $\lv = (\ell_1, \hdots, \ell_n)$ in $\R_+^n$, the univariate exponential polynomial 
\[
f(x) = p(\exp(ix\lv)) = p(e^{ix\ell_1}, \hdots, e^{ix\ell_n}) = \sum_{\balpha}c_{\balpha}e^{ix\langle \lv, \balpha\rangle}
\]
is real-rooted. That is, any point $z\in \C$ with $f(z) = 0$ has imaginary part equal to zero. 
The main result of this paper is to establish the converse, namely that every real-rooted exponential polynomial comes as such a restriction, up to multiplication by non-vanishing exponent. Given a tuple of real numbers, say $\omega_1, \hdots, \omega_s$, denote the dimension (as a $ \Q $-vector space) of their $ \Q $-linear span by $ \dim_{\Q}\{\omega_1, \hdots, \omega_s\} $.  

\begin{thm}\label{thm:main}
Let $f(x) = \sum_{j=0}^s c_j e^{\lambda_j x}$ where $c_0, \hdots, c_s \in \C^{*}$ and $\lambda_0, \hdots, \lambda_s\in \C$, ordered such that $ \im(\lambda_0)=\min_{0\le j\le s}\im(\lambda_j) $. Let $ n=\dim_{\Q}\{\im(\lambda_{1}-\lambda_{0}), \hdots, \im(\lambda_{s}-\lambda_{0})\} $. If $f(x)$ is real rooted, then there is a Lee-Yang polynomial $p\in \C[z_1, \hdots, z_n]$ and $\lv\in \R_+^n$ such that 
\[f(x) = e^{\lambda_{0}x}p(\exp(ix\lv)),\]
and the entries of $\lv$ are $ \Q $-linearly independent. 
\end{thm}
	
We prove this theorem in \Cref{sec:proof}. 	
\begin{remark}
	The construction in Theorem \ref{thm:main} is not unique, however the dimension $ n $ is optimal. If $ q\in\C[z_{1},\ldots,z_{m}] $ and $ \textbf{v}\in\R^{m} $ are such that $ f(x) = e^{\lambda_{0}x}q(\exp(ix\textbf{v})) $, then $ \im(\lambda_{1}-\lambda_{0}), \hdots, \im(\lambda_{s}-\lambda_{0}) $ are in the $ \Q $-linear span of the entries of $ \textbf{v} $, so $ m\ge n. $
\end{remark}
\begin{example}\label{ex:OU} In \cite[Lemma 2]{OlevUlan20}, Olevskii and Ulanovskii show that the function
		\begin{equation*}
			f(x)=\sin(\pi x)+\varepsilon\sin(x)=\frac{1}{2i}\left(e^{i\pi x}-e^{-i\pi x}+\varepsilon e^{ix}-\varepsilon e^{-ix}\right)
		\end{equation*} 
		is real rooted for any real $ \varepsilon $ with $ |\varepsilon|\le\frac{1}{2}$. Here $ (\lambda_{0},\lambda_{1},\lambda_{2},\lambda_{3})=i(-\pi,-1,1,\pi) $. The entries of $(\im(\lambda_1 - \lambda_0), \im(\lambda_2 - \lambda_0), \im(\lambda_3- \lambda_0))  =(\pi-1,\pi+1,2\pi) $ span a two-dimensional vectorspace over $\Q$. We can realize the function $f(x)$, up to an exponential multiple, as the restriction of the Lee-Yang polynomial $ p(z_{1},z_{2})=\frac{1}{2i}\left(z_{1}z_{2}-1+\varepsilon z_{2}-\varepsilon z_{1}\right) $ to $(z_1, z_2) = \exp(ix\lv)$ for $ \lv=(\pi-1,\pi+1) $. 
Namely, 
\[
f(x)=\frac{e^{\lambda_{0}x}}{2i}\left(e^{i2\pi x}-1+\varepsilon e^{i(\pi+1)x}-\varepsilon e^{i(\pi-1)x}\right)=e^{\lambda_{0}x}p(\exp(ix\lv)).
\]
We discuss this example further in \Cref{ex:OUcont} to illustrate the general construction in \Cref{thm:main}. It is worth mentioning that the case $\epsilon=\frac{1}{3}$ also follows from \cite[Equation (43)]{KurasovSarnak20jmp}.  
%The vector $ \lv $ is positive and rationally independent. To see that $ p $ is a Lee-Yang polynomial, notice that $ p(z_{1},z_{2})=0 $ if and only if $ z_{1}=\varphi(z_{2}) $ for $ \varphi(z)=\frac{1-\varepsilon z}{z-\varepsilon} $. A simple check gives $ |\varphi(1)|=|\varphi(i)|=| \varphi(-i)|=1 $ and $  $. Since $ \varphi $ is a Möbius transformation, once can check that $ |\varphi(1)|=|\varphi(i)|=| \varphi(-i)|=1 $ and $ |\varphi(0)|=|\frac{1}{\varepsilon}|\ge 2 $, which is enough to conclude that $ \varphi $ interchange between $ \D $ and $ \D^{-1}=\overline{\C}\setminus\overline{\D} $. In particular, there are no solutions $ z_{1}=\varphi(z_{2}) $ with $ (z_{1},z_{2})\in \D^{2}\cup(\D^{-1})^{2}  $, so $ p $ is a Lee-Yang polynomial.\qed  
	\end{example}
	
Our main motivation for \Cref{thm:main} comes from recent developments in the study of \emph{crystalline measures} and \emph{Fourier Quasicrystals}. To state these, let $ \mathcal{S}(\R) $ denote the space of \emph{Schwartz functions}: smooth functions $ f:\R\to\C $ with $ \sup_{x\in\R}|x^{n}\frac{d^{m}}{d^{m}x}f(x)|<\infty $ for all $ n,m\in\Z_{\geq 0} $. The elements of the dual space $ \mathcal{S}'(\R) $ are called \emph{tempered distributions}, and we focus on these that are also measures. A \emph{tempered measure} $ \mu $ is a complex valued Borel measure on $ \R $ that satisfy $ |\int fd\mu|<\infty $ for all $ f\in\mathcal{S}(\R) $. Such a measure $ \mu $ is said to be \emph{discrete} if there is a discrete (locally finite) set $ \Lambda\subset\R $ and complex coefficients $ (a_{x})_{x\in\Lambda} $ such that $ \int fd\mu=\sum_{x\in\Lambda}a_{x}f(x) $, and its absolute value $ |\mu| $ is the non-negative measure defined by replacing $ a_{x} $ with $ |a_{x}| $ for all $ x\in\Lambda $ (there can by $ \mu\in \mathcal{S}'(\R) $ with $ |\mu|\notin \mathcal{S}'(\R) $). The Fourier transform is a linear automorphism of $ \mathcal{S}(\R) $, and the Fourier transform of $ \mu\in \mathcal{S}'(\R) $ is defined by duality, $ \int fd\hat{\mu}:=\int\hat{f}d\mu $, for all $ f\in\mathcal{S}(\R) $.   

A \emph{crystalline measure}, according to Meyer \cite{Meyer2016}, is a discrete and tempered measure $ \mu\in \mathcal{S}'(\R) $ whose Fourier transform $ \hat{\mu} $ is also discrete. A \emph{Fourier Quasicrystal}(FQ), according to Lev and Olevskii \cite{LevOlev2017fourier}, is a crystalline measure $ \mu $ for which $ |\mu|$ and $ |\hat{\mu}| $ are also tempered. We say that $ \mu $ is $ \N $-valued if $ \mu(I)\in\N $ for every interval $ I\subset\R $, i.e., $ a_{x}\in\N $ for all $ x\in\Lambda $. 

The first example of an $ \N $-valued FQ which is not periodic (or a linear combination of such) was constructed by Kurasov and Sarnak \cite{KurasovSarnak20jmp}. They showed that if $ p(z_{1},\ldots,z_{n}) $ is a Lee-Yang polynomial and  $ \lv\in\R_{+}^{n} $ then the discrete measure $ \mu=\mu_{p,\lv} $, with $ \Lambda=\set{x\in\R\ : \ p(\exp(ix\lv))=0} $ and $ a_{x} $ equal to the multiplicity of $ x $ as a zero of $ p(\exp(ix\lv)) $, is an $ \N $-valued FQ. The inverse question was asked by Olevskii and Ulanovskii \cite{OlevUlan20} who showed that any $ \N $-valued FQ, $ \mu=\sum_{x\in\Lambda}a_{x}\delta_{x} $, can be associated with a real-rooted exponential polynomial $f$, such that $ \Lambda=\set{x\in\R\ : \ f(x)=0} $ and the coefficients $ a_{x} $ are the multiplicities of the zeros of $f $. The results of \cite{KurasovSarnak20jmp,OlevUlan20} together with our Theorem \ref{thm:main} yields:
\begin{cor}\label{cor:main} A measure $\mu$ on $\R$ is an $\N$-valued Fourier quasicrystal if and only if $ \mu=\mu_{p,\lv} $, for some Lee-Yang polynomial $ p(z_{1},\ldots,z_{n}) $ and positive frequencies $ \lv \in \R_+^n$.
\end{cor}
In future work, we establish the connection between properties of $ \N $-valued Fourier quasicrystal and its associated 
Lee-Yang polynomial in more detail. \\

	{\bf Notation.} For vectors ${\bf z} = (z_1, \hdots, z_n)$ and $\balpha = (\alpha_1, \hdots, \alpha_n)$, we use $\z^{\balpha}$ to denote $\prod_{j=1}^nz_{j}^{\alpha_j}$ and $\exp(\z)$ for $(e^{z_1}, \hdots, e^{z_n})$.  For the positive and negative real lines, we use $\R_+ = \{x\in \R: x>0\} $ and $\R_- = \{x\in \R: x<0\} $, respectively. 
\\	
	
	{\bf Acknowledgements.} The work began while the first and third authors were visitors at the Institute for Advanced Studies. 
	They would like to thank Greg Knese, Mihai Putinar, Pavel Kurasov and Peter Sarnak for comments. The first author is supported by the Simons Foundation Grant 601948, DJ.
	The second author is supported by an NSF GRFP fellowship and a Hertz Foundation fellowship. The third author was partially supported by NSF-DMS grant \#2153746 and a Sloan Research Fellowship. 
This material is based upon work directly supported by the National Science Foundation Grant No. DMS-1926686, and indirectly supported by the National Science Foundation Grant No. CCF-1900460.

\section{Amoebas and a reduction}	
In this section, we prove several useful lemmata, including a reduction to purely imaginary exponents and a connection to amoebae, in the sense of Gelfand, Kapranov, and Zelevinsky.

\begin{lem}\label{lem: complex frequencies}
	Let $f(x) =\sum_{j=0}^{s}c_j e^{\lambda_j x}$ where $c_0, \hdots, c_s \in \C^*$ with distinct $\lambda_0, \hdots, \lambda_s \in \C$, ordered such that $ \im(\lambda_{0})\le\im(\lambda_{j}) $ for all $ j \geq 1$. If $ f $ is real-rooted, then $ \lambda_{j}=\lambda_{0}+i\omega_j $, with $ \omega_j\in\R_{+} $, for all $ j\ge1 $. Namely, for all $x\in \C$,
	\[
	f(x) =e^{\lambda_{0}x}\left(c_{0}+\sum_{j=1}^{s}c_j e^{i\omega_j x}\right).
	\]
\end{lem}
\begin{proof}
Let $  \conv(\Lambda) $ be the convex hull of the $ s+1 $ points $\Lambda=   \{\lambda_{0},\lambda_{1},\ldots,\lambda_{s}\} $ in the complex plane $ \C\simeq \R^{2}$. To every edge $ e$ of the polygon $ \conv(\Lambda) $ we associate a normal vector $\hat{n}_{e}\in  \C\simeq \R^{2}$. If $e$ has endpoints $\lambda_j, \lambda_k$, then  $\hat{n}_{e}=i(\lambda_{j}-\lambda_{k}) $.
For each such edge $e$ and $\varepsilon>0$, define the infinite strip
$I_{e}(\varepsilon)=\set{t\hat{n}_{e}+z: t\in\R,\ |z|<\varepsilon }$.
	A century old result by Polya \cite{Polya20zerosexppol} (see also Langer's review paper \cite[Theorem 8]{Langer31}) says that the zero set of $ f $ is unbounded and all but finitely many of the zeros lie in the strips $I_{e}(\varepsilon) $, for suitable choices of $ \varepsilon>0 $, where $ e$ ranges over the edges of $ \conv(\Lambda) $. Furthermore, each of these strips contains infinitely many zeros. In particular, if $ f $ is real-rooted then there can be only one strip and its direction is real, i.e.~$\hat{n}_{e}\in\R $. Therefore $ \conv(\Lambda) $ is a line segment whose normal is real. It follows that for every $j\geq 1$, the numbers $\lambda_j- \lambda_0$ are purely imaginary. 
By our choice of ordering, the values $\omega_j = \im(\lambda_j- \lambda_0) = (\lambda_j- \lambda_0)/i$ are also nonnegative. Note that since $\lambda_0, \hdots, \lambda_s$ are distinct, each $\omega_j$ cannot be zero and so is strictly positive. 
\end{proof}

\begin{lem}\label{lem:Exp2Poly}
Let $f(x) =c_{0}+\sum_{j=1}^{s}c_j e^{i\omega_j x}$ where $c_0, \hdots, c_s \in \C^*$ and $ \omega_1,\ldots\omega_{s}\in\R_{+} $. Let $ n=\dim_{\Q}(\omega_1, \hdots, \omega_s) $. Then there exists a polynomial $p\in \C[z_1, \hdots, z_n]$ and a positive vector $\lv\in \R_+^n$ with $ \Q $-linearly independent entries $ (\dim_{\Q}(\lv)=n) $ for which 
\[
f(x) =p(\exp(ix\lv))
\]
for all $x\in \C$. In particular, $ p({\bf 0})=c_{0}$ which is non-zero by assumption. 
\end{lem}

\begin{proof}
Let $R \subset \Q^s$ denote the set of $\Q$-linear relations on the entries of $\bomega = (\omega_1, \hdots, \omega_s)$ and $L\subset \Q^s$ its orthogonal complement. That is, we define $R = \{{\bf r}\in \Q^s : \langle \bomega, {\bf r}\rangle =0\}$ and $L= R^{\perp}$ to be $\{{\bf v}\in \Q^s: \langle {\bf v}, {\bf r}\rangle=0 \ \forall {\bf r}\in R\}$.    By assumption, $L$ is an $n$-dimensional $\Q$-linear subspace of $\Q^s$ and so $L_{\R} = L\otimes_{\Q}\R = \{{\bf v}\in \R^s: \langle {\bf v}, {\bf r}\rangle=0 \ \forall {\bf r}\in R\}$ is an $n$-dimensional $\R$-linear subspace of $\R^s$.  

Consider the intersection of $L_{\R}$ with the nonnegative orthant $\R_{\geq 0}^s$. This is a rational polyhedral cone of dimension at most $n$ that contains the vector $\bomega$. 
By Carath\'eodory's Theorem, $\bomega$ can be written as a nonnegative combination of some choice of $n$ extreme rays of this cone. In particular, we can choose integer representatives ${\bf v}_1, \hdots, {\bf v}_n\in \Z_{\geq 0}^s$ of these $n$ extreme rays and write  $\bomega= \sum_{j=1}^n \ell_j {\bf v}_j$ for some $\ell_j\geq 0$.  Moreover, since ${\rm span}_{\Q}\{\omega_1, \hdots, \omega_s\}$ has dimension $n$ over $\Q$ and the vectors ${\bf v}_j$ have rational entries, the entries of $\lv = (\ell_1, \hdots, \ell_n)$ must be linearly independent over $\Q$. In particular, none of the values $\ell_j$ can be zero, giving $\lv = (\ell_1, \hdots, \ell_n)\in \R_+^n$. 

Let $A\in \Z_{\geq0}^{n\times s}$ be the $n\times s$ matrix with rows ${\bf v}_1, \hdots, {\bf v}_n$. Then by construction $\bomega = A^T \lv$. Consider the polynomial 
$p(z_1, \hdots, z_n) = c_0+ \sum_{j=1}^s c_j {\bf z}^{\balpha(j)}$,
where  where $\balpha(1), \hdots, \balpha(s)$ are the columns of $A$. 
By construction, for every $j=1, \hdots, s$, we have $\omega_j = \langle \lv, \balpha(j)\rangle$, giving $ (\exp(ix\lv))^{\balpha(j)}=e^{i x\langle \lv, \balpha(j)\rangle}=e^{i \omega_j x} $. Therefore, for any $ x\in\C $, 
\[
p(\exp(ix\lv))= c_0 + \sum_{j=1}^s c_{j}e^{i \omega_j x}.\qedhere\]
\end{proof}
\begin{Def}The \textbf{amoeba} $\A(p)$ of a polynomial $ p(\z)\in \C[z_1, \hdots, z_n]$ is the image of the variety of $p$ under the map $(\C^*)^n \to \R^n$ taking $(z_1, \hdots, z_n)$ to $(\log|z_1|, \hdots, \log|z_n|)$. Equivalently, 
	\[\mathcal{A}(p)=\set{\xv\in\R^{n}\ :\ \exists \btheta\in(\R/2\pi\Z)^{n}\ \text{s.t.}\ p(\exp(\xv+i\btheta))=0 }.\]
\end{Def}
It is a classical theorem by Gelfand, Kapranov, and Zelevinsky that every connected component of the complement $ \mathcal{A}(p)^{c}=\R^{n}\setminus\mathcal{A}(p) $ is open and convex \cite{GKZ}. 
See also \cite{MR2715250, MR1752241}. The real logarithm $ z\mapsto\log|z| $ maps the outer disc $ \C\setminus\overline{\D} $ onto $ \R_{+} $, so $ p $ does not vanish in $ (\C\setminus\overline{\D})^{n} $ if and only if $ \R_{+}^{n} \subset \mathcal{A}(p)^{c} $. Similarly, the real logarithm maps the punctured disk $ \D\setminus\{0\}=\D\cap\C^{*} $ onto $ \R_{-} $, so $ p $ does not vanish in $ (\D\cap\C^{*})^{n} $ if and only if $ \R_{-}^{n} \subset \mathcal{A}(p)^{c} $. If we further ask that $ p(0)\ne0 $, then $p$ has no monomial factors, in which case its zero set in $(\C^*)^n$ is dense in its zero set in $\C^n$ (both in the Euclidean and Zariski topologies on $\C^n$). See e.g.~\cite[Lemma~7.1, p.~121]{Shafarevich2}. The next lemma follows. 
\begin{lem}\label{lem:AmoebaLY}
	A polynomial $ p(\z) \in\C[z_{1},\ldots,z_{n}] $ with $ p(0)\ne 0 $ is Schur stable if and only if $ \R_{-}^{n}\subset \mathcal{A}(p)^{c}$ and is a Lee-Yang polynomial if and only if $ \R_{-}^{n}\cup\R_{+}^{n} \subset \mathcal{A}(p)^{c}$.
\end{lem}

The amoeba of a polynomial $p = \sum_{\balpha} c_{\balpha} {\bf z}^{\balpha}$ has a close connection to its Newton polytope 
\[ {\rm Newt}(p) = \conv({\rm supp}(p)) \ \text{ where } \ {\rm supp}(p)=\{ \balpha\in \Z^n : c_{\balpha}\neq 0\}. 
\]
In particular, every connected component $ E $ of the complement $\mathcal{A}(p)^c$ can be associated to integer point 
 $\bbeta\in{\rm Newt}(p)\cap\Z^{n}$, and the recession cone of such a convex region $ E $ coincides with the normal cone of ${\rm Newt}(p)$ at $\bbeta$, 
 \[
 \mathcal{N}_{\bbeta} = \left\{{\bf w}\in \R^n : \langle \bbeta, {\bf w} \rangle = \max_{\balpha \in {\rm Newt}(p)}  \langle \balpha, {\bf w} \rangle \right\}.
 \]
See \cite[Proposition 2.6]{MR1752241}. For the sake of completeness, we provide a proof of the following simplification: 

\begin{prop}[see Proposition 2.6 of \cite{MR1752241}]\label{prop:recCone}
Let $p = \sum_{\balpha} c_{\balpha}\z^{\balpha}\in \C[z_1, \hdots, z_n]$ and let $\lv\in \R^n$ with $ \Q $-linearly independent entries.
\begin{enumerate}
	\item There is a unique vertex $\bbeta\in{\rm Newt}(p)\cap\Z^{n}$ that maximizes $\langle \bbeta, \lv \rangle=\max_{\balpha\in{\rm Newt}(p)}\langle \balpha, \lv \rangle$.
	\item Suppose that the ray $ \R_{+}\lv=\set{t\lv\ :\ t>0}  $ lies in $\mathcal{A}(p)^c$, and let $ E $ denote the connected component of $\mathcal{A}(p)^c$ containing $\R_+\lv$. Then $ E $ contains the interior of the normal cone at $ \bbeta $, $ {\rm int}(\mathcal{N}_{\bbeta})\subset E $.
\end{enumerate}
\end{prop}

\begin{proof}
Since ${\rm Newt}(p)$ is convex and compact, then the maximum of $ \balpha\mapsto \langle \balpha, \lv \rangle $ is attained at a vertex. The map $ \balpha\mapsto \langle \balpha, \lv \rangle $ is injective on the vertices of ${\rm Newt}(p)$, as the vertices are integer and the entries of $ \lv $ are $ \Q $-linearly independent. This proves (1). 

Let $ \bbeta $ be the unique maximizer. Since its a vertex of  ${\rm Newt}(p)$, then  the normal cone of ${\rm Newt}(p)$ at $\bbeta$, 
\[
\mathcal{N}_{\bbeta} = \left\{{\bf w}\in \R^n : \langle \bbeta, {\bf w} \rangle = \max_{\balpha \in {\rm Newt}(p)}  \langle \balpha, {\bf w} \rangle \right\},
\]
is a full-dimensional polyhedral cone. Consider the open convex polyhedron 
\[P_{\bbeta} = \left\{{\bf w}\in \R^n : \langle \bbeta, {\bf w} \rangle + \log|c_{\bbeta}| > \max_{\balpha\neq \bbeta}\left(\langle \balpha, {\bf w} \rangle + \log|c_{\balpha}| + \log M \right) \right\},\quad M=|{\rm supp}(p)|.
\]
We claim that $P_{\bbeta}\subseteq E$.  To see this, note that for any ${\bf w}\in P_{\bbeta}$, the strict inequality $|c_{\bbeta}e^{\langle\bbeta, {\bf w}\rangle}| > \max_{\balpha\neq \bbeta} M|c_{\balpha} e^{\langle\balpha, {\bf w} \rangle}|\ge \sum_{\balpha\neq \bbeta} |c_{\balpha} e^{\langle\balpha, {\bf w} \rangle}|$ holds, and therefore $p(\exp({\bf w} + i \btheta))\ne0$ for any $\btheta\in (\R/2\pi \Z)^n$, so ${\bf w}\not\in \mathcal{A}(p)$. Therefore $P_{\bbeta}\subseteq A(p)^c$. 
Recall that $ \langle \bbeta, \lv \rangle > \max_{\balpha\neq \bbeta}\langle \balpha, \lv \rangle $, so for sufficiently large $t\in \R$, $t\lv\in P_{\bbeta}$, showing that $E$ is the connected component of $A(p)^c$ containing $P_{\bbeta}$. 

Now suppose that ${\bf w}\in{\rm int}(\mathcal{N}_{\bbeta})$, and let ${\bf \tilde{w}} = {\bf w} - \varepsilon \lv$, for sufficiently small $\varepsilon>0$, so that $ {\bf \tilde{w}}\in {\rm int}(\mathcal{N}_{\bbeta})$ as well. 
For large enough $t>1$, $t{\bf \tilde{w}}$ belongs to $P_{\bbeta}\subseteq E$. 
Since $\R_+\lv\subseteq E$ and $E$ is convex, we see that it also contains
${\bf w} = (1-\frac{1}{t})(\frac{\varepsilon t}{t-1} \lv) + \frac{1}{t}(t{\bf \tilde{w}})$.
\end{proof}

\begin{lem}\label{lem:MonomialChange}
Let $p\in \C[z_1, \hdots, z_n]$, with $p({\bf 0})\neq 0$ and $ \lv\in\R_{+}^{n}$ with $ \Q $-linearly independent entries. If the set $\{t\lv : t\in \R^*\}$ is disjoint from the 
amoeba of $p$, then there exists a Lee-Yang polynomial $q\in \C[z_1, \hdots, z_n]$ and vector 
$\tilde{\lv}\in\R_{+}^{n}$ with $ \Q $-linearly independent entries for which
	\[p(\exp(ix \lv))=q(\exp(ix\tilde{\lv})).\]
\end{lem}
\begin{proof}
Let $E_+$ and $E_{-}$ denote the connected components of $\A(p)^c$ containing $\R_+\lv$ and $\R_-\lv$, respectively. By \Cref{prop:recCone} with $\lv$ (resp.~$ -\lv $), the connected component $E_+$ (resp. $E_-$) contains a full-dimensional rational polyhedral cone $ C_{+} $ (resp. $ C_{-} $) whose interior contains $\lv$ (resp. $-\lv$). Let $C$ denote the intersection of the rational polyhedral cones $C_+$ and $-C_-$ and the nonnegative orthant $\R_{\geq 0}^n$. Then $C$ is a pointed full-dimensional rational polyhedral cone containing $\lv$. Moreover the amoeba of $p$ is disjoint from ${\rm int}(C) \cup {\rm int}(-C)$. 

The same argument as in the proof of Lemma \ref{lem:Exp2Poly} allows to choose extreme rays of $C$ which are integer vectors, ${\bf v}_1, \hdots, {\bf v}_n \in C\cap \Z^n$, such that $\lv = \sum_{j=1}^n \tilde{\ell}_j {\bf v}_j$ with $\tilde{\ell}_j\geq 0$. Since the entries of $\lv$ are $ \Q $-linearly independent, we conclude that $ \tilde{\ell}_j> 0 $ for all $ j $, giving $\tilde{\lv}\in \R_+^n$. Let $A$ denote the $n\times n$ matrix with rows ${\bf v}_1, \hdots, {\bf v}_n$, so that $A^{T}\tilde{\lv} = \lv$. Let $ \balpha(1), \hdots, \balpha(n) $ be the columns of $ A $, and consider the polynomial $q(z_1, \hdots, z_n) = p({\bf z}^A) = p({\bf z}^{\balpha(1)}, \hdots, {\bf z}^{\balpha(n)})$. For any $\xv, \btheta\in \R^n$ and any $\balpha(k)=(a_{1k}, \hdots, a_{nk})$ , 
\[
\exp(\xv+i\btheta)^{\balpha(k)}  =  \prod_{j=1}^n e^{(x_j+i\theta_j)a_{jk}}  = e^{(A^{T}(\xv+i\btheta))_{k}}.\]
From this we see that 
$q(\exp(\xv+i\btheta))
 =p(\exp(A^{T}(\xv+i\btheta) ))$ and that 
the amoeba of $q$ is the set of $ \xv\in\R^{n} $ such that $ A^{T}\xv $ belongs to the amoeba of $p$. 

By construction, any conic combination $\sum_{j=1}^n x_j {\bf v}_j$ with $x_j> 0$ belongs to the interior of the cone $C$. Thus, $ A^{T}\xv\in {\rm int}(C)\subset \mathcal{A}(p)^c$ for any $ \xv\in\R_{+}^n$, which means that  $\R_{+}^n\subset \mathcal{A}(q)^c$. The same argument applies to $ \R_{-}^{n} $ and $ {\rm int}(-C)$, so $\A(q)$ is disjoint from $\R_+^n\cup \R_-^n$. Since $ q({\bf 0})=p({\bf 0})\ne0 $, then $q$ is a Lee-Yang polynomial by \Cref{lem:AmoebaLY}.
Finally note that, 
\[
q(\exp(i x \tilde{\lv})) 
= p(\exp(i x A^{T}\tilde{\lv})
=
p(\exp(ix\lv)),
\]
 for any $x\in \C$.\end{proof}

\begin{figure}
\includegraphics[height=2in]{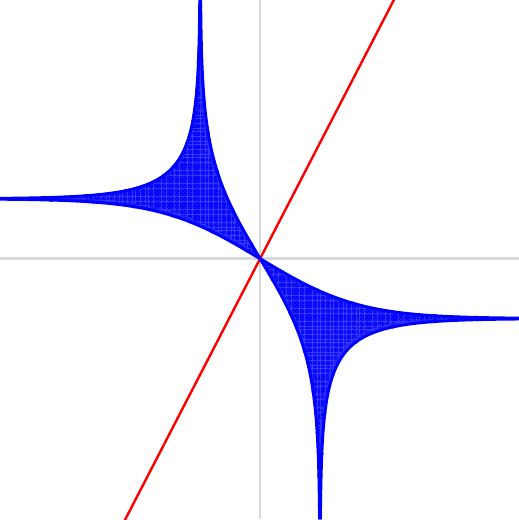}
\caption{The amoeba $\mathcal{A}(p)$ and line spanned by  $\lv$ from \Cref{ex:OUcont}.}\label{fig:amoeba}
\end{figure}

\begin{example}[\Cref{ex:OU} cont']\label{ex:OUcont}
Consider the real-rooted exponential function 
\[
e^{-\lambda_0x}f(x)=e^{i\pi x}(\sin(\pi x)+\varepsilon\sin(x))=
\frac{1}{2i}\left(-1-\varepsilon e^{i(\pi - 1)x} +\varepsilon e^{i(\pi + 1)x}+ e^{i2\pi x}\right)
\]
discussed in \Cref{ex:OU}. 
The entries of $\bomega =(\pi-1,\pi+1,2\pi) $ span a two-dimensional vector space over $\Q$. Following the proof of \Cref{lem:Exp2Poly}, 
the set $R$ of $\Q$ linear relations on $\bomega = (\omega_1, \omega_2, \omega_3) = (\pi-1,\pi+1,2\pi) $ is spanned by ${\bf r} = (1,1,-1)$. 
If $L = R^{\perp}$ then $L\otimes_\Q\R$ is $L_{\R} = \{(v_1,v_2,v_3)\in \R^3 : v_1+v_2-v_3 = 0\}$. 
The intersection of $L_{\R}$ with the nonnegative orthant is a two-dimensional cone with extreme rays spanned by ${\bf v}_1 = (1,0,1)$ and ${\bf v}_2 = (0,1,1)$.
Then $\bomega = \ell_1{\bf v}_1 + \ell_2{\bf v}_2$ for the positive vector 
$\lv = (\ell_1, \ell_2) = (\pi - 1, \pi+1)\in \R_+^2$. 

Take the matrix $A =  {\small\begin{pmatrix}  1 & 0 & 1 \\ 0 & 1 & 1\end{pmatrix}}$ with rows ${\bf v}_1, {\bf v}_2$ and columns $A_1, A_2, A_3 \in \Z_{\geq 0}^2$. We construct $p(z_1, z_2)$ by replacing $e^{i\omega_jx}$ in the expression above with $(z_1, z_2)^{A_j}$, which gives
\[
p(z_1, z_2) = \frac{1}{2i}\left(-1-\varepsilon z_1 +\varepsilon z_2+ z_1z_2\right).
\]
By construction, $e^{-\lambda_0x}f(x) = p(\exp(ix\lv))$. 
The amoeba of $p$ for $\varepsilon= 1/4$ is shown in \Cref{fig:amoeba}. 
As we will see in \Cref{sec:AlmostPeriodic} below, the real rootedness of $p(\exp(ix\lv))$ guarantees that the amoeba  $\mathcal{A}(p)$ is disjoint from $\R_-\lv \cup \R_+\lv$. 

The Newton polytope of $p$ is the unit square $[0,1]^2$. The vertex maximized by $\balpha \mapsto \langle \balpha, \lv\rangle$ is $(1,1)$ and the normal cone of ${\rm Newt}(p)$ at the vertex $ (1,1) $ is $\R_{\geq 0}^2$. Similarly, the vertex $(0,0)$ achieves the maximum inner product with $-\lv$ and the normal cone at $(0,0)$ is $\R_{\leq 0}^2$. By \Cref{prop:recCone}, it follows that $\R_-^2\cup \R_+^2$ are disjoint from the amoeba $\mathcal{A}(p)$. Since $p(0)=\frac{1}{2i}\ne0$, we conclude by \Cref{lem:AmoebaLY} that $p$ is a Lee-Yang polynomial. 
\end{example}

\section{Almost Periodic Functions}\label{sec:AlmostPeriodic}

Exponential sums are intimately related with the notion of \emph{almost periodic functions}, developed by 
Harald Bohr\footnote{It is worth mentioning that Harald Bohr was the younger brother of Niels Bohr, one of the founding fathers of quantum theory and the 1922 physics Nobel laureate.}. 
He defined this notion for functions on the real line \cite[p.32]{Bohr2018almost} and proved the ``Fundamental Theorem of Almost Periodic Functions'' \cite[p.80]{Bohr2018almost}, which states that the space of almost periodic functions is the closure of  the set of exponential polynomials with purely imaginary exponents (real frequencies) in the space of bounded continuous functions from $ \R $ to $ \C $.
Here we need the following slight generalization.

\begin{Def}
Given an open strip $ I_{h}=\set{x+iy\ :\ x\in\R, |y|< h} $ in the complex plane, a holomorphic function $ f $ on $ I_h $ is said to be \textbf{almost periodic function on $ I_h $} if for any $ \varepsilon>0 $ there is a relatively dense set $ T_{\varepsilon}\subset\R $ such that $ |f(z+\tau)-f(z)|<\varepsilon $ for any $ \tau\in T_{\varepsilon}  $ and any $ z\in I_{h} $. By relatively dense, we mean that there exists some $ R=R(\varepsilon)>0 $ such that $ T_{\varepsilon}$ has nonempty intersection with every interval of length $R$ in the real line.
\end{Def}

According to \cite[Theorem 1, p. 266]{Levin_zeros}, if $ f(x)=\sum_{j=0}^{s}c_{j}e^{i\omega_j x} $ where $c_0, \hdots, c_s \in \C^*$ and $\omega_0, \hdots, \omega_s \in \R$, then for any $h>0$, $ f_{h}(x):=f(x+ih) $ is almost periodic in the upper half plane, and so $ f $ is almost periodic in the strip $ I_{h} $. With this, we may restate \cite[Lemma 1, p. 268]{Levin_zeros} for exponential polynomials instead of almost periodic functions.

\begin{lem}\cite[Lemma 1, p. 268]{Levin_zeros}\label{lem: lower bound}
	Suppose that $ f(x)=\sum_{j=0}^{s}c_{j}e^{i\omega_j x} $, where $c_0, \hdots, c_s \in \C^*$ and $\omega_0, \hdots, \omega_s \in \R$. Given any $ h>0 $ and $ \delta>0 $ there is a uniform lower bound 
	\[|f(x+iy)|\ge m(h,\delta)>0,\]
	for all $ x\in\R $ and $ y\in (-h,h) $ such that $ |x+iy-z|>\delta $ for every zero $ z\in\C $ with $ f(z)=0 $.
\end{lem}     
\begin{cor}\label{cor:Bohr}
If $\lv\in \R^n$ has rationally independent entries and $p\in \C[z_1, \hdots, z_n]$ has the property 
that $f(x)=p(\exp(ix\lv))$ is real rooted, then for every $t\in \R^*$ and $\btheta\in (\R/2\pi\Z)^n$, $p(\exp(t\lv+i \btheta))\ne0$. That is, the set $\{t\lv : t\in \R^*\}$ is disjoint from the amoeba $\A(p)$ of $p$. 
\end{cor}
\begin{proof}
Let $ t\in \R^*$ and let $ f_{t}(x):=f(x+it)=p(\exp(-t\lv+ix\lv)) $. The zero set $ f_{t} $ lies on the line $ \im(z)=-t $ since $ f $ was real rooted. Fix arbitrary $h>0 $ and $ 0<\delta<|t| $, so that \cref{lem: lower bound} gives
\begin{equation*}
	|p(\exp(t\lv+ix\lv))|=|f_{t}(x)|\ge m(h,\delta)>0,
\end{equation*}
uniformly over $ x\in\R $, since $ |x-z|\ge|t|>\delta $ for every zero $ f_{t}(z)=0 $. Since the entries of $ \lv $ are $ \Q $-linearly independent, the set $ \set{\exp(t\lv+ix\lv) : x\in\R} $ is dense in $ \set{\exp(t\lv+i\btheta): \btheta\in(\R/2\pi\Z)^{n}} $ in the Euclidean topology on $\C^n$. By the continuity of $ p $,
\[|p(\exp(t\lv+i\btheta))|\ge m(h,\delta)>0,\]
for any $ \btheta\in(\R/2\pi\Z)^{n} $. 
\end{proof}

\section{Proof of the main theorem}\label{sec:proof}
\begin{proof}[Proof of \Cref{thm:main}]Let $f(x) = \sum_{j=0}^s c_j e^{\lambda_j x}$ where $c_0, \hdots, c_s \in \C^{*}$ and $\lambda_0, \hdots, \lambda_s\in \C$,  where $ \im(\lambda_0)\le\im(\lambda_{j}) $ for all $ j\geq 1$ and suppose that $f(x)$ is real rooted. By \Cref{lem: complex frequencies},
	\[
f(x) =e^{\lambda_{0}x}\left(c_{0}+\sum_{j=1}^{s}c_j e^{i\omega_j x}\right)
\]
with real positive frequencies $ \omega_j=\im(\lambda_{j}-\lambda_{0})>0 $ for $ j\ge1 $. Let $n=\dim_{\Q}(\omega_1, \hdots, \omega_s)$. By \Cref{lem:Exp2Poly}, there is a polynomial $p\in \C[z_1, \hdots, z_n]$ with constant coefficient $c_0\neq 0$ and vector $\lv\in \R_+^n$ with rationally independent entries for which
\[f(x) = e^{\lambda_{0}x}p(\exp(ix\lv)).\]
Since $f$ is real rooted, then by \Cref{cor:Bohr}, the set $\{t\lv : t\in \R^*\}$ is disjoint from the amoeba $\A(p)$ of $p$. 
Together with \Cref{lem:MonomialChange} this guarantees the existence of a Lee-Yang polynomial 
$q\in \C[z_1, \hdots, z_n]$ and vector $\tilde{\lv}\in \R_+^n$ for which 
\[q(\exp(ix\tilde{\lv})) = p(\exp(ix\lv)) = e^{-\lambda_{0}x}f(x). \qedhere
\]
\end{proof}
%
%\section{Possible Generalizations?}
%
%
%\begin{question}
%Does a general real-rooted exponential polynomial $f(x) = \sum_{j=1}^s a_j(x)e^{i \omega_j x}$ come as the 
%restriction $f(x) = p(x, e^{i \ell_2 x}, \hdots, e^{i \ell_n x})$ of a polynomial $p$ stable to $\H_+\times \D^{n-1}$ and $\H_-\times (\C\backslash \overline{\D})^{n-1}$? 
%\end{question}
%
%\begin{question}
%What if instead of real-rooted we ask about functions  $f(x) = \sum_{j=1}^s e^{i \omega_j x}$ with no zeros in the upper half plane? Do those all come from the restrictions of $\D$-stable polynomials? 
%\end{question}
%
%

\bibliographystyle{plain}
\bibliography{ExpPoly}

\end{document}